\documentclass[times]{anzsauth}
\usepackage{moreverb}
\usepackage[colorlinks,bookmarksopen,bookmarksnumbered,citecolor=red,urlcolor=red]{hyperref}
\def\volumeyear{2013}
\usepackage{amsmath}
\usepackage{amssymb,nccbbb}
\usepackage{floatflt}

\newtheorem{thm}{Theorem}[section]
\newtheorem{lemma}[thm]{Lemma}
\newtheorem{prop}[thm]{Proposition}

\newcommand{\sign}{\mathrm{sign}}
\newcommand{\mean}{\mathsf{E}}
\newcommand{\var}{\mathsf{D}}

\begin{document}
\runningheads{Estimation of the Exponent of SL}{O.~Yanushkevichiene and V.~V.~Saenko}
\title{Estimation  of the characteristic exponent of stable laws}

\author{O. Yanushkevichiene\addressnum{1}, V. V. Saenko\addressnum{2}\corrauth}
\affiliation{Institute of Mathematics and Informatics, Vilnius University\\Lithuanian University of Educational Science\\
Ulyanovsk State University}
\address{\addressnum{1}Institute of Mathematics and Informatics, Vilnius University, Akademijos, 4, Vilnius LT - 08663, Lithuania,\\
 \addressnum{1}Lithuanian University of Educational Science, Studentu 39, Vilnius LT-08106, Lithuania, (e-mail:olgjan@zebra.lt),\\
 \addressnum{2}Ulyanovsk State University, Leo Tolstoy str., 42, Ulyanovsk, 432970, Russia (e-mail:saenkovv@gmail.com)}

\acks{The work is completed by Russian Foundation for basic Research (grants 13-01-00585, 12-01-33074, 12-01-00660) and Ministry of Education and Science of the Russian Federation}


\begin{abstract}
A statistical algorithm for estimating the characteristic parameter $\alpha$ of the
stable law  is presented and the estimate of its quadratic deviation
is obtained in the paper. This algorithm is applied in the description of the fluctuation induced flux in the edge region of the plasma cord.
\end{abstract}

\maketitle

\section{Introduction}
The importance of investigating stable laws stems from the fact that all kinds of
limit distributions of linearly-normed sums of an infinitely increasing number of independent
identically distributed random variable are stable laws. This fact stipulates the
application of stable laws in the problems of economics, physics, and other spheres of science.
The latter makes the problem of estimating the parameters of stable laws very important.

In the book of \cite{Zolotarev:AMS:86} the major features of the stable law are described.
A class of stable distributions comprises a four parametric family of the functions
$g(x;\alpha,\beta,\gamma,\lambda)$ with the parameters
$$
0<\alpha\leqslant 2,\quad -1\leqslant\beta\leqslant1,\quad -\infty<\gamma<\infty,\quad\lambda>0
$$
(we shall call the parameter values admissible). This class of distributions is usually defined
by the corresponding set of characteristic functions that can be written, e.g. in the following
way (see \cite{Zolotarev:AMS:86})
$$
\hat g(k;\alpha,\beta,\lambda,\gamma)=\left\{\begin{array}{ll}
\exp\left\{ \lambda\left( ik\gamma-|k|^\alpha+
i|k|^\alpha\beta\tan\left(\frac{\pi}{2}\alpha \right)\sign k \right)\right\},& \alpha\neq1\\
\exp\left\{ \lambda\left( ik\gamma-|k|^\alpha-
ik\beta\frac{2}{\pi}\log|k| \right)\right\},&\alpha=1
\end{array}\right.,
$$
The class of stable distributions includes the normal distribution $(\alpha=2)$, Cauchy distribution
$(\alpha=1)$ and the Levy distribution $(\alpha=1/2)$ together with that symmetric to it ($\alpha=1/2,\beta=-1$).
Naturally, those variants of stable laws that appear due to a change in the values of unfixed parameters
$\gamma$ and $\lambda$ should also be attributed to them. In other cases we did not succeed in
expressing the density of stable distributions by elementary functions.

The problem of parameter estimation of stable laws is hampered by the absence of explicit
expressions of densities. However, there are already a lot of works devoted to this problem. In our work
the exponent $\alpha$ of stable laws is estimated. The estimation is based on the method, proposed
by V. Zolotarev, for evaluating the parameters $\gamma,\theta,\tau$ of strictly stable random variables.
The estimation of quadratic deviation of the constructed estimate has been obtained.

Before stating the main problem of this work we recall some results obtained in the book of
\cite{Zolotarev:AMS:86}. Let there be a sample $X_1, X_2,\dots,X_n$ of independent
identically distributed random variables (i.i.d.r.v.) having a strictly stable distribution with
the characteristic function
\begin{equation}\label{eq:formE}
\hat g(k;\nu,\theta,\tau)=\exp\left\{\hspace{-0.5mm}-\hspace{-0.5mm}\exp\left\{ \nu^{-1/2}\hspace{-1mm} \left( \log|k|\hspace{-0.5mm}+\hspace{-0.5mm}\tau\hspace{-0.5mm}-\hspace{-0.5mm}i\frac{\pi}{2}\theta\sign k \right)+
\bbbc\left( \nu^{-1/2}\hspace{-0.5mm}-\hspace{-0.5mm}1 \right)\right\}\right\},
\end{equation}
where $\bbbc=0.577\dots$ is the Euler constant and the parameters are varying within the limits of the
interval $\nu\geqslant1/4$, $|\theta|\leqslant\min(1,2\sqrt{\nu}-1)$, $\tau<\infty$. According
to this sample we construct two sets of independent identically distributed (within the framework of each
sample) random variables $U_i=\sign Y_i$ and $V_i=\log|Y_i|$, $i=1,\dots,n$. Then the estimator
\begin{equation}\label{eq:ssl_est}
\tilde\theta=A_U,\quad \tilde\tau=A_V, \quad
\end{equation}
\begin{equation}\label{eq:nu_est}
\tilde\nu=\max\left (\hat\nu,\left (1+|\tilde\theta| \right)^2/4 \right)
\end{equation}
will be statistical estimates $\tilde \nu, \tilde\theta$ and $\tilde\tau$ of the parameters
$\nu,\theta,$ and $\tau$, where
\begin{equation}\label{eq:nuest0}
\hat\nu=\frac{6}{\pi^2}B^2_V-\frac{3}{2}B^2_U+1,
\end{equation}
and
$$
A_X=\frac{1}{n}\sum_{j=1}^nX_j, \quad B_X^2=\frac{1}{n-1}\sum_{j=1}^n\left(X_j-A_X \right)^2
$$
are the sample mean and variance.

The basic results of this work are as follows:
\begin{thm}\label{theo:alpha_var}
The estimate $\tilde\alpha=1/\sqrt{\tilde\nu}$ of the parameter $\alpha$ is asymptotically unbiased of
order $1/\sqrt{n}$ and its quadratic deviation satisfies the inequality
\begin{align}
\mean\left(\hat\alpha-\alpha\right)^2&\leqslant\frac{\mean(\hat\nu-\nu)^2}{4\nu^3}+
\frac{12}{n\nu^{3/2}}\sqrt{\mean(\hat\nu-\nu)^2}\left(P_3^2+\frac{\sqrt{P_3'}}{n^{1/8}}\right)\nonumber\\
&+144\left(P_3^4/n^2+P_3'/n^{9/4}\right),\label{eq:alpha_vat_main}
\end{align}
where the expressions $P_3$, $P_3'$ will be defined later.
\end{thm}

The results obtained considerably improved the estimate of quadratic deviation of the parameter $\alpha$
obtained in \cite{Yanushkevichiene:LtMathRin:81}. In addition, an experimental research has been performed here.

\section{Algorithm for estimating the parameters of strictly stable laws}\label{sec:est}

\begin{prop}\label{prop:nu_est}
Transformation (\ref{eq:nu_est}) can only decrease the quadratic deviation of the estimate from the
real value of the parameter.
\end{prop}
\begin{proof}

If $\hat\nu\geqslant(1+|\tilde\theta|)^2/4$, then $\tilde\nu=\hat\nu$ and the proposition
is obvious. If $\hat\nu<(1+|\tilde\theta|)^2/4$, then $\tilde\nu=(1+|\tilde\theta|)^2/4$. Taking
into consideration that, for a given value $\theta$, the domain of values of the parameter $\nu$ is defined from
the condition $\nu\geqslant(1+|\theta|)^2/4$, we obtain the inequality
$$
|\tilde\nu-\nu|<|\hat\nu-\nu|,
$$
which finally proves the proposition.
\end{proof}

Next, let there be a sequence $Y_1,Y_2,\dots,Y_n$ of i.i.d.r.v's having the stable distribution with
the characteristic function
$$
\hat g(k;\alpha,\beta,\lambda,\gamma)=\left\{\begin{array}{ll}
\exp\left\{ \lambda\left( ik\gamma-|k|^\alpha+i|k|^\alpha\beta\tan\left(\frac{\pi}{2}\alpha \right)\sign k \right)\right\},& \alpha\neq1\\
\exp\left\{ \lambda\left( ik\gamma-|k|^\alpha-
ik\beta\frac{2}{\pi}\log|k| \right)\right\},&\alpha=1
\end{array}\right.,
$$
where the parameters are varying  within scope $0<\alpha\leqslant2$, $-1\leqslant\beta\leqslant1$, $-\infty<\gamma<\infty$,
$\lambda>0$. The parameters $\nu,\theta,\tau$ are connected with the parameter
$\alpha,\beta,\lambda,\gamma$ by the expressions
\begin{eqnarray}
\nu&=&\alpha^{-2},\nonumber\\
\theta&=&\left\{\begin{array}{ll}
\displaystyle \frac{2}{\pi\alpha}\arctan\left(\beta\tan\left(\frac{\pi\alpha}{2}\right)\right),
&\nu\neq1,\\[0.5cm]
\displaystyle\frac{2}{\pi}\arctan(\gamma),&\nu=1,
\end{array}\right.\nonumber\\
\tau&=&\left\{\begin{array}{ll}
\displaystyle(1/\alpha)\left[\log\lambda-\log\left(\cos\left(\arctan\left(\beta\tan(\pi\alpha/2)\right)\right)\right)-
\bbbc(\alpha-1)\right],&\nu\neq1,\\
\log\lambda-\log\left(\cos(\arctan\gamma)\right),&\nu=1.
\end{array}\right.\nonumber
\end{eqnarray}

Construction of the estimates $\tilde\alpha,\tilde\beta,\tilde\lambda,\tilde\gamma$ of the parameters
$\alpha, \beta, \lambda, \gamma$ is based on the application of estimates (\ref{eq:ssl_est}) and (\ref{eq:nu_est}).
To this end, the initial sequence $Y_1,Y_2,\dots,Y_{n}$ should be transformed so that the
random variables, included in the new sample $Y_1',Y_2',\dots,Y_{m}'$, belonged to the class of strictly
stable laws. In the book of \cite{Zolotarev:AMS:86}, the transformation
\begin{equation}\label{eq:transform}
Y_j'=Y_{3j-2}-1/2(Y_{3j-1}+Y_{2j}),\quad j=1,2,\dots, 3n.
\end{equation}
is selected as this kind of transformation. As a result of the transformation, the random variables
$Y'_j$ have the distribution $g(x;\alpha,\beta',\lambda',\gamma')$, where the parameters $\beta',\lambda',\gamma'$
are connected with that of the distribution of the initial sample through the relations
\begin{eqnarray}
\beta'&=&\frac{1-2^{1-\alpha}}{1+2^{1-\alpha}}\beta,\nonumber\\
\lambda'&=&\left(1+2^{1-\alpha}\right)\lambda,\nonumber\\
\gamma'&=&\left\{\begin{array}{cl}
0,&\mbox{если} \alpha\neq1,\\
-\frac{\log2}{\pi}\beta, &\mbox{если} \alpha=1,
\end{array}\right.\nonumber
\end{eqnarray}
and the characteristic exponent of the stable law does not change. Hence we can see that, if $\alpha\neq1$,
then $\gamma=0$, and if $\alpha=1$, then $\beta=0$. It means that the random variable $Y'$ belongs to the class
of strictly stable laws and its characteristic function can be represented in the form (\ref{eq:formE}).
This fact allows us to apply algorithm (\ref{eq:ssl_est}) to the sample $Y_j'$ and to obtain estimates
$\tilde\nu',\tilde \theta', \tilde\tau'$. It ought to be noted here that application of transformation
(\ref{eq:transform}) leads to contraction of the domain of change of the parameters $\theta',\tau'$.
In order to estimate the parameters of the stable law, we choose the following expressions
\begin{equation}\label{eq:est2}
\tilde\nu^*=\tilde\nu',\quad \tilde\theta^*=\min(H(\tilde\nu'),|\tilde\theta'|)\sign\tilde\theta',
\quad \tilde\tau^*=\tilde\tau',
\end{equation}
where the estimates $\tilde\nu',\tilde \theta', \tilde\tau'$ are obtained by the sequence $Y_1',\dots,Y_m'$
with the aid of formulas (\ref{eq:ssl_est}) and (\ref{eq:nu_est}), and
$$
H(t)=\left\{ \begin{array}{ll}
 \frac{2}{\pi}\sqrt{t}\arctan\left(\frac{1-\exp\left\{ (1-1/\sqrt{t})\log2 \right\}}{1+\exp\left\{ (1-1/\sqrt{t})\log2 \right\}}
\tan\left( \frac{\pi}{2\sqrt{t}} \right) \right),& t\neq1\\
\frac{2}{\pi}\arctan\left(\frac{\log2}{\pi} \right),& t=1.
\end{array}
 \right.
$$
\begin{prop}
Transformation of the parameters $\theta$ in (\ref{eq:est2}) can only diminish the value of mean-square
deviation of the estimate $\tilde\theta^*$.
\end{prop}
The proof is analogous to that of Proposition~\ref{prop:nu_est}.

We obtain the following expressions
\begin{eqnarray}
 \tilde\alpha&=&1/\sqrt{\tilde\nu^*}\label{eq:alpha_est}\\
\tilde\beta&=&\left\{\begin{array}{cl}
    \displaystyle\left(\frac{1+2^{1-\tilde\alpha}}{1-2^{1-\tilde\alpha}}\right)
    \frac{\tan\left(\frac{\pi}{2}\tilde\theta^*\tilde\alpha \right)}
    {\tan\left(\frac{\pi}{2} \tilde\alpha\right)},&\tilde\alpha\neq1\\
    \displaystyle-\frac{\pi}{\log2}\tan\left(\frac{\pi}{2}\tilde\theta^* \right),&\tilde\alpha=1
\end{array}\right.\nonumber\\
\tilde\lambda&=&\left\{\begin{array}{cl}
 \displaystyle\frac{\cos\left(\frac{\pi}{2}\tilde\theta^*\tilde\alpha\right)}
{1+2^{1-\tilde\alpha}}
\exp\left\{\tilde\tau^*\tilde\alpha+\bbbc\left( \tilde\alpha-1 \right) \right\}, &\tilde\alpha\neq1\\
\displaystyle(1/2)\exp\{\tilde\tau^*\}\cos\left(\pi\tilde\theta^*/2\right),&\tilde\alpha=1
\end{array}
\right.\nonumber
\end{eqnarray}
for the estimates $\tilde\alpha,\tilde\beta,\tilde\lambda$.

\section{Evaluation of mean-square deviation of the estimate $\tilde\alpha$}

We will focus our attention on obtaining the estimate of mean-square deviation of $\tilde\alpha$.
The idea of finding the estimate consist in the expansion of expression (\ref{eq:alpha_est})
in a Taylor series with respect to the parameter $\tilde\nu$ in the neighborhood of the true value $\nu$.
However, before stating a theorem for estimating the mean-square deviation, we shall prove a number of lemmas.
\begin{lemma}\label{lemma:lnV}
Let $X$ be a  strictly stable random variable. Then the following expressions hold.
\begin{eqnarray}
\mean\log|X|&=&\tau,\qquad \mean\log^2|X|=\tau^2+\var V,\nonumber\\
\mean\log^3|X|&=&\tau^3+3\tau\var V+2\zeta(3)\left(\nu^{3/2}-1\right),\nonumber\\
\mean\log^4|X|&=&\tau^4+6\tau^2\var V+8\tau\zeta(3)\left(\nu^{3/2}-1\right)+3(\var V)^2\nonumber\\
&+&\pi^4\left(\left(1-\theta^4\right)/8 +\left(\nu^2-1\right)/15 \right),\nonumber\\
\mean\log^5|X|&=&\tau^5+10\tau^3\var V+20\tau^2\zeta(3)\left(\nu^{3/2}-1\right)+15\tau(\var V)^2\nonumber\\
&+& 5\tau\pi^4\left(\frac{1-\theta^4}{8} +\frac{\nu^2-1}{15} \right)+20\zeta(3)\left(\nu^{3/2}-1\right)\var V\nonumber\\
&+&24\zeta(5)\left(\nu^{5/2}-1\right),\nonumber\\
\mean\log^6|X|&=&\tau^6+15\tau^4\var V+40\tau^3\zeta(3)\left(\nu^{3/2}-1\right)+45\tau^2(\var V)^2+15(\var V)^3\nonumber\\
&+& 15\tau^2\pi^4\left(\frac{1-\theta^4}{8} +\frac{\nu^2-1}{15} \right)+
120\tau\zeta(3)\left(\nu^{3/2}-1\right)\var V\nonumber\\
&+&40\zeta^2(3)\left(\nu^{3/2}-1\right)^2+ 144\tau \zeta(5)\left(\nu^{5/2}-1\right)\nonumber\\
&+&15\pi^4\left(\frac{1-\theta^4}{8} +\frac{\nu^2-1}{15} \right)\var V
+8\pi^6\left(\frac{1-\theta^6}{32} +\frac{\nu^3-1}{63} \right),\nonumber
\end{eqnarray}
\begin{eqnarray}
\mean\log^7|X|&=\hspace{-2mm}&\tau^7\!+\!21\tau^5\var V\!+\! 70\tau^4\zeta(3)\left(\nu^{3/2}-1\right)\!+\!
35\tau^3\pi^4\left(\frac{1\hspace{-0.5mm}-\hspace{-0.5mm}\theta^4}{8}\hspace{-0.5mm} +
\hspace{-0.5mm}\frac{\nu^2-1}{15} \right)\nonumber\\
&+&\hspace{-2mm}105\tau(\var V)^3+420\tau^2\zeta(3)\left(\nu^{3/2}-1\right)\var V+ 504\tau^2\zeta(5)\left(\nu^{5/2}-1\right)\nonumber\\
&+&\hspace{-2mm}210\zeta(3)\left(\nu^{3/2}-1\right)(\var V)^2+105\tau\pi^4\left(\frac{1-\theta^4}{8} +\frac{\nu^2-1}{15} \right)\var V\nonumber\\
&+&\hspace{-2mm}56\tau\pi^6\left(\frac{1-\theta^6}{32} +\frac{\nu^3-1}{63} \right)+504\zeta(5)\left(\nu^{5/2}-1\right)\var V\nonumber\\
&+&\hspace{-2mm}70\zeta(3)\left(\nu^{3/2}\!-\!1\right)\pi^4\!\left(\frac{1\hspace{-0.5mm}-\hspace{-0.5mm}\theta^4}{8} +\frac{\nu^2-1}{15} \right)+
280\tau\zeta^2(3)\left(\nu^{3/2}\hspace{-0.5mm}-\hspace{-0.5mm}1\right)^2\nonumber\\
&+&\hspace{-2mm}720\zeta(7)\left(\nu^{7/2}-1\right)+105\tau^3(\var V)^2\nonumber
\end{eqnarray}
\begin{eqnarray}
\mean\log^8|X|&=&\hspace{-2mm}\tau^8+28\tau^6\var V+112\tau^5\zeta(3)\left(\nu^{3/2}-1\right)+210\tau^4(\var V)^2\nonumber\\
&+&\hspace{-2mm}70\tau^4\pi^4\left(\frac{1-\theta^4}{8} +\frac{\nu^2-1}{15} \right)+420\tau^2(\var V)^3 \nonumber\\ &+&\hspace{-2mm}1120\tau^3\zeta(3)\left(\nu^{3/2}-1\right)\var V
+105(\var V)^4+1344\tau^3\zeta(5)\left(\nu^{5/2}-1\right)\nonumber\\
&+&\hspace{-2mm}420\tau^2\pi^4\left(\frac{1-\theta^4}{8} +\frac{\nu^2-1}{15} \right)\var V
+1680\tau(\var V)^2\zeta(3)\left(\nu^{3/2}-1\right)\nonumber\\
&+&\hspace{-2mm}224\tau^2\pi^6\left(\frac{1-\theta^6}{32} +\frac{\nu^3-1}{63} \right)
+1120\tau^2\zeta^2(3)\left(\nu^{3/2}-1\right)^2\nonumber\\
&+&\hspace{-2mm}210(\var V)^2\pi^4\left(\frac{1-\theta^4}{8} +\frac{\nu^2-1}{15} \right)\nonumber\\
&+&\hspace{-2mm}560\tau\zeta(3)\left(\nu^{3/2}-1\right)\pi^4\left(\frac{1-\theta^4}{8} +\frac{\nu^2-1}{15} \right)\nonumber\\
&+&\hspace{-2mm}1120\zeta^2(3)\left(\nu^{3/2}-1\right)^2\var V
+4032\tau\zeta(5)\left(\nu^{5/2}-1\right)\var V\nonumber\\
&+&\hspace{-2mm}35\pi^8\left(\frac{1-\theta^4}{8} +\frac{\nu^2-1}{15} \right)^2
+5760\tau\zeta(7)\left(\nu^{7/2}-1\right)\nonumber\\
&+&\hspace{-2mm}224\pi^6\left(\frac{1-\theta^6}{32} +\frac{\nu^3-1}{63} \right)\var V\nonumber\\
&+&\hspace{-2mm}2688\zeta(3)\left(\nu^{3/2}-1\right)\zeta(5)\left(\nu^{5/2}-1\right)\nonumber\\
&+&\hspace{-2mm}16\pi^8\left(\frac{17(1-\theta^8)}{256} +\frac{\nu^4-1}{30} \right)\nonumber,
\end{eqnarray}

Here $\var V=\frac{\pi^2}{4}\left(1-\theta^2 \right)+\frac{\pi^2}{6}(\nu-1)$ and $\zeta(x)$ is the zeta function.
\end{lemma}
\begin{proof} It follows from the Theorem~3.6.1 of \cite{Zolotarev:AMS:86}.
\end{proof}
Let us introduce the notations $\bar U=U-\mean U$, $\bar V=V-\mean V$.
Applying expression (\ref{eq:ssl_est}), it is easy to show that
$$
\mean U=\theta,\quad\mean V=\tau.
$$

The following lemma is valid.
\begin{lemma}
For centered moments of the random variable $U$ the equalities
\begin{equation}\label{eq:Umoment}
\begin{split}
 \mean\bar U&=0,\quad
\mean\bar U^2=1-\theta^2,\\
\mean\bar U^3&=2\theta(\theta^2-1),\quad
\mean\bar U^4=1+2\theta^2-3\theta^4,\\
\mean\bar U^5&=4\theta(\theta^4-1),\quad
\mean\bar U^6=1+9\theta^2-5\theta^4-5\theta^6,\\
\mean\bar U^7&=2\theta(3\theta^6+7\theta^4-7\theta^2-3),\\
\mean\bar U^8&=1+20\theta^2+14\theta^4-28\theta^6-7\theta^8,
\end{split}
\end{equation}
hold.
\end{lemma}
\begin{proof} Consider, for example, the fifth central moment. Taking into consideration that
$$
\sign^nU=\left\{\begin{array}{ll}
 \sign U,&\mbox{if}\ n\ \mbox{is odd}\\
1,&\mbox{if}\ n\ \mbox{is even},
\end{array}
\right.,
$$
we get
\begin{eqnarray}
 \mean (U&-&\theta)^5=\mean\left[ \sign^5 U-5\theta\sign^4U+10\theta^2\sign^3U-
10\theta^3\sign^2U\right.\nonumber\\
&+&\left.5\theta^4\sign U-\theta^5 \right]
=\theta-5\theta+10\theta^3-10\theta^3+5\theta^5-\theta^5=
4\theta(\theta^4-1).\nonumber
\end{eqnarray}
The remaining equalities are proved by analogously.
\end{proof}

\begin{lemma}
The expressions
\begin{equation}\label{eq:Vmoment}
\begin{split}
\mean\bar V&=0,\quad
\mean\bar V^2=\frac{\pi^2}{4}\left(1-\theta^4\right)+
\frac{\pi^2}{6}\left(\nu-1\right)=\var V,\\
\mean\bar V^3&=2\zeta(3)\left(\nu^{3/2} -1\right),\quad
\mean\bar V^4=\pi^4\left(\frac{1-\theta^4}{8}+\frac{\nu^2-1}{15}\right)
+3(\var V)^2,\\
\mean\bar V^5&=20\zeta(3)\left(\nu^{3/2} -1\right)\var V+
24\zeta(5)\left(\nu^{5/2} -1\right),\\
\mean\bar V^6&=15(\var V)^3+15\pi^4\left(\frac{1-\theta^4}{8}+\frac{\nu^2-1}{15}\right)\var V
+8\pi^6\left(\frac{1-\theta^6}{32}+\frac{\nu^2-1}{63}\right)\\
&+40\zeta^2(3)\left(\nu^{3/2} -1\right)^2,\\
\mean\bar V^7&=105(\var V)^4+210(\var V)^2\pi^4\left(\frac{1-\theta^4}{8}+\frac{\nu^2-1}{15}\right)\\
&+2688\zeta(3)\left(\nu^{3/2} -1\right)\zeta(5)\left(\nu^{5/2} -1\right)\\
&+8\pi^8\left(\frac{17(1-\theta^8)}{128}+\frac{\nu^4-1}{15}\right).
\end{split}
\end{equation}
are valid for the central moments of the random variable $V$
\end{lemma}

\begin{proof} Let us consider, for instance, the fifth central moment. Making use of the results of
Lemma~\ref{lemma:lnV} we get
\begin{eqnarray}
 \mean\bar V^5&=&\mean\left(V-\mean V \right)^5=\mean V^5-5\mean V^4\tau+
10\mean V^3\tau^2\hspace{-0.5mm}-\hspace{-0.5mm}10\mean V^2\tau^3\hspace{-0.5mm}+ \hspace{-0.5mm}5\mean V\tau^4\hspace{-0.5mm}-\hspace{-0.5mm}\tau^5\nonumber\\
&=&\tau^5\!+\!10\tau^3\var V\!+\!20\tau^2\zeta(3)\left(\nu^{3/2}\! -\!1\right)\!+\!15\tau(\var V)^2\nonumber\\
&+&5\tau\pi^4\left(\frac{1-\theta^4}{8}+\frac{\nu^2-1}{15}\right)
+20\zeta(3)\left(\nu^{3/2} -1\right)\var V\nonumber\\
&+&24\zeta(5)\left(\nu^{5/2}\hspace{-0.5mm} -\hspace{-0.5mm}1\right)\!-\!5\tau^5\!-\!30\tau^3\var V
\hspace{-0.5mm}-\hspace{-0.5mm}40\tau^2\zeta(3)\left(\nu^{3/2}\hspace{-0.5mm} -\hspace{-0.5mm}1\right)\hspace{-0.5mm}-\hspace{-0.5mm}15\tau(\var V)^2\nonumber\\
&-&5\tau\pi^4\left(\frac{1-\theta^4}{8}+\frac{\nu^2-1}{15}\right)+10\tau^5+30\tau^3\var V
+20\tau^2\zeta(3)\left(\nu^{3/2} -1\right)\nonumber\\
&-&10\tau^5-10\tau^3\var V+5\tau^5-\tau^5\nonumber\\
&=&20\zeta(3)\left(\nu^{3/2} -1\right)\var V+
24\zeta(5)\left(\nu^{5/2} -1\right).\nonumber
\end{eqnarray}
All the remaining  equalities are obtained similarly\footnote{To obtain the higher order central moments of the r.v's $\bar U$ and $\bar V$, we can use the Theorem~4.1.2 from the book of \cite{Zolotarev:AMS:86}}.
\end{proof}

Let us introduce the notations:
\begin{eqnarray}
 P_1&=&16\left( 15\theta^2-71\theta^4+26\theta^6-45\theta^8 \right)+
\frac{12}{n-1}\left( 5-96\theta^2+302\theta^4-356\theta^6\right.\nonumber\\
&+&\left.155\theta^8 \right)
=\frac{4}{(n-1)^2}\left( 5-52\theta^2+134\theta^4+132\theta^6+45\theta^8\right).\nonumber\\
P_2&=&\hspace{-2mm}4\left\{ 18 (\var V)^4\!+\!39(\var V)^2\pi^4\left(\frac{1-\theta^4}{8}\!+\!\frac{\nu^2-1}{15}\right)\! +\! 96\zeta^2(3)\left(\nu^{3/2}\hspace{-0.5mm} -\hspace{-0.5mm}1\right)^2\!\var V\right.\nonumber\\
&+&\hspace{-2mm}8\pi^8\left(\frac{1-\theta^4}{8}+\frac{\nu^2-1}{15}\right)^2+
48\pi^6\left(\frac{1-\theta^6}{32}+\frac{\nu^3-1}{63}\right)\var V\nonumber\\
&+&\hspace{-2mm}\left.384\zeta(3)\left(\nu^{3/2} -1\right)\zeta(5)\left(\nu^{5/2} -1\right)+
2\pi^8\left(\frac{17(1-\theta^6)}{128}+\frac{\nu^4-1}{15}\right)\right\}\nonumber\\
&+&\hspace{-2mm}\frac{12}{n-1}\left\{15(\var V)^4+25(\var V)^2\pi^4\left(\frac{1-\theta^4}{8}+\frac{\nu^2-1}{15}\right) \right.\nonumber\\
&+&\hspace{-2mm}16\pi^6\left(\frac{1-\theta^6}{32}+\frac{\nu^3-1}{63}\right)\var V
+32\zeta^2(3)\left(\nu^{3/2} -1\right)^2\var V \nonumber\\ &+&\hspace{-2mm}\left.2\pi^8\left(\frac{1-\theta^4}{8}+\frac{\nu^2-1}{15}\right)^2\right\}
+\frac{12}{(n-1)^2}\left\{-32\zeta^2(3)\left(\nu^{3/2} -1\right)^2\var V\right.\nonumber\\
&-&\hspace{-2mm}128\zeta(3)\left(\nu^{3/2} -1\right)\zeta(5)\left(\nu^{5/2} -1\right)
+12(\var V)^2\pi^4\left(\frac{1-\theta^4}{8}+\frac{\nu^2-1}{15}\right)\nonumber\\
&+&\hspace{-2mm}\left.13(\var V)^4\right\}
+\frac{8}{(n-1)^3}\left\{\pi^8\left(\frac{1-\theta^4}{8}+\frac{\nu^2-1}{15}\right)^2+6(\var V)^4\right.\nonumber\\
&-&\hspace{-2mm}\left.48\zeta^2(3)\left(\nu^{3/2} -1\right)^2\var V\right\},\nonumber\\
P_3&=&3\sqrt[4]{3}\sqrt{\theta^2(1-\theta^2)}+\frac{6\sqrt[4]{3}}{\pi^2}\sqrt{\pi^4\left(\frac{1-\theta^4}{8}+\frac{\nu^2-1}{15}\right)+
3(\var V)^2},\label{eq:p3}\\
P_3'&=&n^{1/4}\left( P_3+\frac{3}{2}\left(\frac{|P_1|}{n} \right)^{1/4} +\frac{6}{\pi^2}\left(\frac{|P_2|}{n} \right)^{1/4}\right)^4 -n^{1/4}P_3^4.\label{eq:p3'}
\end{eqnarray}

\begin{lemma}\label{lemma:b4}
The inequality
\begin{equation}\label{eq:nu_4moment}
\mean\left(\tilde\nu-\nu\right)^4\leqslant\frac{P_3^4}{n^2}+\frac{P_3'}{n^{9/4}}
\end{equation}
is true.
\end{lemma}
\begin{proof} By combining (\ref{eq:nuest0}) and assumption~\ref{prop:nu_est}, we get
\begin{equation}\label{eq:inequlity}
\begin{split}
\mean\left(\tilde\nu\right.&-\left.\nu\right)^4\leqslant\mean\left(\hat\nu-\nu\right)^4=
\mean\left(\frac{6}{\pi^2}S_V^2-\frac{3}{2}S_U^2+1 -\mean\left( \frac{6}{\pi^2}S_V^2-
\frac{3}{2}S_U^2+1 \right)\right)^4\\
&=\frac{6^4}{\pi^8}\mean\left( B_V^2-\mean B_V^2 \right)^4
-\frac{6^4}{\pi^6}\mean\left\{ \left( B_V^2-\mean B_V^2 \right)^3\left( B_U^2-\mean B_U^2 \right) \right\}\\
&+\frac{6\cdot3^4}{\pi^4}\mean\left\{ \left( B_V^2-\mean B_V^2 \right)^2\left( B_U^2-\mean B_U^2 \right)^2 \right\}\\
&-\frac{3^4}{\pi^2}\mean\left\{ \left( B_V^2-\mean B_V^2 \right)\left( B_U^2-\mean B_U^2 \right)^3 \right\}\\
&+ \frac{3^4}{2^4}\mean\left( B_U^2-\mean B_U^2 \right)^4
\leqslant\frac{6^4}{\pi^8}\mean\left( B_V^2-\mean B_V^2 \right)^4\\
&-\frac{6^4}{\pi^6}\left(\mean\left( B_V^2-\mean B_V^2 \right)^4\right)^{3/4}
\left(\mean\left( B_U^2-\mean B_U^2 \right)^4 \right)^{1/4}\\
&+\frac{6\cdot3^4}{\pi^4}\left(\mean \left( B_V^2-\mean B_V^2 \right)^4\right)^{1/2}
\left(\mean\left( B_U^2-\mean B_U^2 \right)^2 \right)^{1/2}\\
&-\frac{3^4}{\pi^2}\left(\mean \left( B_V^2-\mean B_V^2 \right)^4\right)^{1/4}
\left(\mean\left( B_U^2-\mean B_U^2 \right)^4 \right)^{3/4}+\frac{3^4}{2^4}\mean\left( B_U^2-\mean B_U^2 \right)^4\\
&=\left( \frac{6}{\pi^2}\sqrt[4]{\mean\left( B_V^2-\mean B_V^2 \right)^4}+
\frac{3}{2}\sqrt[4]{\mean\left( B_U^2-\mean B_U^2 \right)^4} \right)^4
\end{split}
\end{equation}
Thus we need to estimate the mathematical expectations $\mean\left( B_V^2-\mean B_V^2 \right)^4$ and
$\mean\left( B_U^2-\mean B_U^2 \right)^4$.

Let $X_i,\ i=1,2,\dots,n$ are the independent random variables with mathematical expectation $a$ and variance
$b^2$. The equality
\begin{align*}
\mean\left( B_X^2\right.&-\left.\mean B_X^2 \right)^4=\mean\left[ \frac{n}{n-1}\left( \frac{1}{n}\sum_{i=1}^nX_i^2-
\left(\frac{1}{n}\sum_{i=1}^nX_i\right)^2 \right)-b^2 \right]^4\\
=&\mean\left[ \frac{n}{n-1}\frac{1}{n^2}\left( \sum_{i=1}^nX_i^2(n-1)-\sum_{i\neq j}X_iX_j \right)-b^2 \right]^4\\
=&\mean\left[ \frac{1}{n(n-1)}\left( \sum_{i=1}^n(X_i-a)^2(n-1)-\sum_{i\neq j}(X_i-a)(X_j-a) \right)-b^2 \right]^4.
\end{align*}
holds.

Let $Z_i=X_i-a$, then
$\mean Z_i=0$, and $\mean Z_i^2=b^2$. Consequently, we get
\begin{align*}
 \lefteqn{\mean\left( B_Z^2-\mean B_Z^2 \right)^4=\mean\left[ \frac{1}{n}\sum_{i=1}^n\left( Z_i^2-b^2 \right)-
\frac{1}{n(n-1)}\sum_{i\neq j}Z_iZ_j \right]^4}\\
&=\frac{1}{n^4}\mean\left[\sum_{i=1}^n\left( Z_i^2-b^2 \right) \right]^4
=\frac{4}{n^4(n-1)}\mean\left[ \left( \sum_{i=1}^n\left( Z_i^2-b^2 \right) \right)^3\sum_{i\neq j}Z_iZ_j \right]\\
&+\frac{6}{n^4(n-1)^2}\mean\left[ \left( \sum_{i=1}^n\left( Z_i^2-b^2 \right) \right)^2\left(\sum_{i\neq j}Z_iZ_j\right)^2 \right]\\
&-\frac{4}{n^4(n-1)^3}\mean\left[\!\! \left( \sum_{i=1}^n\left( Z_i^2-b^2 \right) \right)\!\!\!\left(\sum_{i\neq j}Z_iZ_j\right)^3 \right]+
\frac{1}{n^4(n-1)^4}\mean\left[ \sum_{i\neq j}Z_iZ_j \right]^4\\
&=\frac{1}{n^4}n\mean\left(Z_i^2-b^2 \right)^4 +\frac{3n(n-1)}{n^4}\left(  \mean\left(Z_i^2-b^2 \right)^2\right)^2
\nonumber\\
&-\frac{4}{n^4(n-1)}\mean\left(\left( \sum_{i=1}^n\left( Z_i^2-b^2 \right)^3
+3\sum_{i\neq j}\left( Z_i^2-b^2 \right)^2\left( Z_j^2-b^2 \right)\right.\right.\\
&+\left.\left.\sum_{i\neq j\neq k}\left( Z_i^2-b^2 \right)
\left( Z_j^2-b^2 \right)\left( Z_k^2-b^2 \right)\right)\sum_{i\neq j} Z_i Z_j\right)\\
&+\frac{6}{n^4(n-1)^2}\mean\left(\! \left( \sum_{i=1}^n\left( Z_i^2-b^2 \right)^2\!+
\sum_{i\neq j}\left( Z_i^2-b^2 \right)\left( Z_j^2-b^2 \right) \right)\!\left( 2\sum_{i_\neq j} Z_i^2Z_j^2+\right.\right.\\
&+\left.\left.4\sum_{i\neq j\neq k}Z_i^2Z_jZ_k +\sum_{i\neq j\neq k\neq m}Z_iZ_jZ_kZ_m\right) \right)\\
&-\frac{4}{n^4(n-1)^3}\mean\left( \sum_{i=1}^n\left( Z_i^2-b^2 \right)\left( 4\sum_{i\neq j}Z_i^3Z_j^3+\right.\right.
24\sum_{i\neq j\neq k}Z_i^3Z_j^2Z_k\\
&+8\sum_{i\neq j\neq k}Z_i^2Z_j^2Z_k^2
+8\hspace{-0.3cm}\sum_{i\neq j\neq k\neq l}\hspace{-0.3cm}Z_i^3Z_jZ_kZ_l+
30\hspace{-0.3cm}\sum_{i\neq j\neq k\neq l}\hspace{-0.3cm}Z_i^2Z_j^2Z_kZ_l\\
&\left.\left.+12\hspace{-0.5cm}\sum_{i\neq j\neq k\neq l\neq m}\hspace{-0.5cm}Z_i^2Z_jZ_kZ_lZ_m+
\hspace{-0.4cm}\sum_{i\neq j\neq k\neq l\neq m\neq p}\hspace{-0.5cm}Z_iZ_jZ_kZ_lZ_mZ_p \right) \right)\\
&+\frac{1}{n^4(n-1)^4}\mean\left( 8\sum_{i\neq j} Z_i^4 Z_j^4+48\sum_{i\neq j\neq k} Z_i^4 Z_j^2Z_k^2+
96\sum_{i\neq j\neq k} Z_i^3 Z_j^3 Z_k^2\right.\\
&+\left.60\hspace{-3mm}\sum_{i\neq j\neq k\neq l}\hspace{-3mm} Z_i^2 Z_j^2Z_k^2 Z_l^2\right).\\
\end{align*}
Taking into consideration that $\mean Z=0$, we get $\mean\sum_{i\neq j\neq k}Z_i^aZ_j^bZ_k=0$.
It is easy to obtain, that
\begin{align*}
\lefteqn{\mean\left(B_Z^2-b^2\right)^4=\frac{1}{n^3}\left(\mean Z^8-4\mean Z^6 b^2+6\mean Z^4b^4-
4\mean Z^2b^2+b^8\right)}&\\
&+\frac{3(n-1)}{n^3}\left(\mean Z^4-b^4\right)^2-
\frac{24}{n^4(n-1)}\mean\left(\sum_{i\neq j}Z_i\left(Z_i^2-b^2\right)^2Z_j\left(Z_j^2-b^2\right)\right)\\
&+\frac{6}{n^4(n-1)^2}\left(\left(\sum_iZ_i^4-2b^2\sum_i Z_i^2-nb^4\right)\right.\times\\
&\times\left(2\sum_{i\neq j}Z_i^2Z_j^2+4\sum_{i\neq j\neq k}Z_i^2Z_jZ_k\right)+
4\sum_{i\neq j}Z_i^2\left(Z_i^2-b^2\right)Z_j^2\left(Z_j^2-b^2\right)\\
&+\left.8\hspace{-2mm}\sum_{i\neq j\neq k}\hspace{-2mm}Z_i^2Z_j\left(Z_i^2-b^2\right)\!\!\left(Z_k^2-b^2\right)Z_k\hspace{-1mm}\right)\!-\!
\frac{4}{n^4(n-1)^3}\mean\!\left(\!8\!\sum_{i\neq j}Z_i^3Z_j^3\left(Z_i^2-b^2\right)\right.\\
&\left.+24\sum_{i\neq j\neq k}Z_i^3Z_j^2Z_k\left(Z_k^2-b^2\right)+
24\sum_{i\neq j\neq k}Z_i^2Z_j^2Z_k^2\left(Z_k^2-b^2\right)\right)+\frac{8\left(\mean Z^4\right)^2}{n^3(n-1)^3}\\
&+\frac{48(n-2)\mean Z^4 b^4}{n^3(n-1)^3}+
\frac{96(n-2)\left(\mean Z^3\right)^2b^2}{n^3(n-1)^3}+\frac{60(n-2)(n-3)b^8}{n^3(n-1)^3}\\
&=\frac{3\left(\mean Z^4-b^4\right)^2}{n^2}-\frac{3\left(\mean Z^4-b^4\right)^2}{n^3}
+\frac{1}{n^3}\left(\mean Z^8-4b^2\mean Z^6+6b^4\mean Z^4-3b^8\right)\\
&-\frac{24}{n^3} \mean Z^3\left(\mean Z^5-2b^2\mean Z^3\right)+\frac{6}{n^4(n-1)^2}\mean\left(4\sum_{i\neq j}Z_i^6Z_j^2+
2\sum_{i\neq j\neq k}Z_i^4Z_j^2Z_k^2\right.\\
&-8b^2\!\sum_{i\neq j}\!Z_i^4Z_j^2\!
-\!4b^2\hspace{-2mm}\sum_{i\neq j\neq k}\hspace{-2mm}Z_i^2Z_j^2Z_k^2\!-\!
2b^4n\!\sum_{i\neq j}\! Z_i^2Z_j^2\!+\!4(n-1)\sum_i\left(\mean Z_i^4-b^4\right)^2\\
&\left.+8(n-1)(n-2)\sum_i\left(\mean Z_i^3\right)^2\right)-\frac{32}{n^3(n-1)^2}\left(
\mean Z^3\mean Z^5-b^2\left(\mean Z^3\right)^2\right)\\
&-\frac{96(n-2)}{n^3(n-1)^2}b^2\left(\mean Z^3\right)^2
+\frac{60(n-3)}{n^3(n-1)^2}b^8-\frac{60(n-3)}{n^3(n-1)^3}b^8+\frac{96b^2\left(\mean Z^3\right)^2}{n^3(n-1)^2}\\
&-\frac{96b^2\left(\mean Z^3\right)^2}{n^3(n-1)^3}+\frac{48b^4\mean Z^4}{n^3(n-1)^2}-\frac{48b^4\mean Z^4}{n^3(n-1)^3}
+\frac{8\left(\mean Z^4\right)^2}{n^3(n-1)^3}=\frac{3\left(\mean Z^4-b^4\right)^2}{n^2}\\
&+\frac{1}{n^3}\!\left(\mean Z^8\!-\!4b^2\mean Z^6\!-\!24\mean Z^3\mean Z^5\!-\!3\left(\mean Z^4\right)^2\!+\!24b^4\mean Z^4\!+\!96b^2\left(\mean Z^3\right)^2\!-\!18b^8\right)\\
&+\frac{12}{n^3(n-1)}\left(2b^2\mean Z^6+
2\left(\mean Z^4\right)^2-17b^4\mean Z^4-12b^2\left(\mean Z^3\right)^2+18b^8\right)\\
&+\frac{4}{n^3(n-1)^2}\left(-8\mean Z^3\mean Z^5+36b^4\mean Z^4+56b^2\left(\mean Z^3\right)^2-69b^8\right)\\
&+\frac{8}{n^3(n-1)^3}\left(\left(\mean Z^4\right)^2-6b^4\mean Z^4-
12b^2\left(\mean Z^3\right)^2+15b^8\right).
\end{align*}
By substituting now (\ref{eq:Vmoment}) and (\ref{eq:Umoment}) into the expression, we obtain
\begin{align*}
\mean\left(B_U^2\right.&-\left.\mean B_U^2\right)^4=\frac{48\theta^4\left(1-\theta^2\right)^2}{n^2}+
\frac{16}{n^3}\left(15\theta^2-71\theta^4+26\theta^6-45\theta^8\right)+\\
+&\frac{12}{n^3(n-1)}\left(5-96\theta^2+302\theta^4-356\theta^6+155\theta^8\right)+\\
+&\frac{4}{n^3(n-1)^2}\left(-33+436\theta^2-1238\theta^4+1300\theta^6-465\theta^8\right)+\\
+&\frac{16}{n^3(n-1)^3}\left(5-52\theta^2+134\theta^4-132\theta^6+45\theta^8\right)=
\frac{48\theta^4(1-\theta^2)^2}{n^2}+\frac{P_1}{n^3}
\end{align*}

\begin{align*}
\mean\left(B_V^2-\right.&\left.\mean B_V^2\right)=\frac{3}{n^2}\left(\frac{\pi^4\left(1-\theta^4\right)}{8}+
\frac{\pi^4\left(\nu^2-1\right)}{15}+3\var V^2\right)^2+\frac{4}{n^3}\biggl( 18\var V^4+\\
+&39\var V^2\pi^4\left(\frac{1-\theta^4}{8}+\frac{\nu^2-1}{15}\right)+
96\var V\zeta^2(3)\left(\nu^{3/2}-1\right)\\
&+8\pi^8\left(\frac{1-\theta^4}{8}+\frac{\nu^2-1}{15}\right)^2
+48\var V\pi^6\left(\frac{1-\theta^6}{32}+\frac{\nu^3-1}{63}\right)\\
&+384\zeta(3)\left(\nu^{3/2}-1\right)\zeta(5)\left(\nu^{5/2}-1\right)
+2\pi^8\left(\frac{17\left(1-\theta^8\right)}{128}\right.\\
&+\!\left.\left.\frac{\left(\nu^4-1\right)}{15}\right)\!\right)\!+\!
\frac{12}{n^3(n-1)}\!\left(\!15\var V^4\!+\!25\var V^2\pi^4\left(\frac{1-\theta^4}{8}+\frac{\nu^2-1}{15}\right)\right.\\
+&16\var V\pi^6\left(\frac{1-\theta^6}{32}+\frac{\nu^3-1}{63}\right)+
32\var V\zeta^2(3)\left(\nu^{3/2}-1\right)^2\\
+&\left.2\pi^8\left(\frac{1-\theta^4}{8}+\frac{\nu^2-1}{15}\right)^2\right)
+\frac{12}{n^3(n-1)^2}\biggl(-32\var V\zeta^2(3)\left(\nu^{3/2}-1\right)^2\\
-&128\zeta(3)\left(\nu^{3/2}-1\right)\zeta(5)\left(\nu^{5/2}-1\right)
+\left.12\var V\pi^4\left(\frac{1-\theta^4}{8}+\frac{\nu^2-1}{15}\right)\right.\\
+&\left.13\var V^4\right)+
\frac{8}{n^3(n-1)^3}\left(\pi^8\left(\frac{1-\theta^4}{8}+\frac{\nu^2-1}{15}\right)^2
+6\var V^4\right.\\
-&48\var V\zeta^2(3)\!\left(\nu^{3/2}\!-\!1\right)^2\!\Biggr)\!=\!
\frac{3}{n^2}\!\left(\!\frac{\pi^4\left(1\!-\!\theta^4\right)}{8}\!+\!
\frac{\pi^4\left(\nu^2\!-\!1\right)}{15}\!+\!3\var V^2\!\!\right)^2\hspace{-2mm}+\!\frac{P_2}{n^3}
\end{align*}
The statement of the lemma easily follows by substituting the expressions obtained into inequality (\ref{eq:inequlity})
and the fact that
$$
(A+B)^{1/4}\leqslant|A|^{1/4}+|B|^{1/4},
$$
we arrive at .
\end{proof}
\begin{lemma}
The inequality
\begin{equation}\label{eq:nu_3moment}
\mean\left|\tilde\nu-\nu\right|^3\leqslant\frac{\sqrt{\mean\left(\hat\nu-\nu\right)^2}}{n}
\left(P_3^2+\frac{\sqrt{P_3'}}{n^{1/8}}\right)
\end{equation}
holds.
\end{lemma}
\begin{proof} The statement of the lemma follows immediately from the Lemma~\ref{lemma:b4} and the
inequalities
$$
(A+B)^{1/2}\leqslant|A|^{1/2}+|B|^{1/2},
$$
and
$$
\mean\left|\tilde\nu-\nu\right|^3\leqslant\mean\left|\hat\nu-\nu\right|^3\leqslant
\sqrt{\mean\left(\tilde\nu-\nu\right)^2\mean\left(\tilde\nu-\nu\right)^4}.
$$

The expressions for $\mean\left(\tilde\nu-\nu\right)^2$ was obtained by \cite{Zolotarev:AMS:86}
\begin{multline}\label{eq:nu_2moment}
\mean\left(\tilde\nu\hspace{-0.5mm}-\hspace{-0.5mm}\nu\right)^2\hspace{-0.5mm}=
\hspace{-0.5mm}\frac{1}{n}\left(\frac{22}{5}(\nu\hspace{-0.5mm}-\hspace{-0.5mm}1)^2\hspace{-0.5mm}+
\hspace{-0.5mm}\frac{6}{5}\left(9\hspace{-0.5mm}-\hspace{-0.5mm}5\theta^2\right)
(\nu\hspace{-0.5mm}-\hspace{-0.5mm}1)+3\left(1\hspace{-0.5mm}-\hspace{-0.5mm}\theta^2\right)
\left(3\hspace{-0.5mm}+\hspace{-0.5mm}\theta^2\right)\right)+\\
+\frac{1}{n(n-1)}\left(2(\nu-1)^2+6\left(1-\theta^2\right)(\nu-1)+9\left(1-\theta^2\right)^2\right)
\end{multline}
By substituting (\ref{eq:nu_2moment}) and (\ref{eq:nu_4moment}) in the previous expression and
making some transformations, we get the statement of the Lemma.
\end{proof}

\paragraph{Proof of the theorem~\ref{theo:alpha_var}.} We can expand (\ref{eq:alpha_est}) in a series with
respect to $\tilde\nu$ in the neighborhood of the true value of the parameter $\nu$
\begin{equation}\label{eq:alpha_expansion1}
\tilde\alpha=\frac{1}{\sqrt{\nu}}-\frac{1}{2}\frac{(\tilde\nu-\nu)}{\nu^{3/2}}+
\frac{3}{8}(\tilde\nu-\nu)^2f(\zeta)^{5/2},
\end{equation}
where $f(\zeta)=1/(\nu-\zeta(\tilde\nu-\nu))$. The latter summand is a remainder term in the Legandre form
and $0<\zeta<1$. Taking into account that the domaing of definition the parameter
$\tilde\nu$ is determined by the condition $\tilde\nu\geqslant1/4$, we obtain
$$
\underset{\zeta,\tilde\nu,\nu}{\max}f(\zeta)=4.
$$

\begin{figure}
\centering
\includegraphics[width=0.9\textwidth]{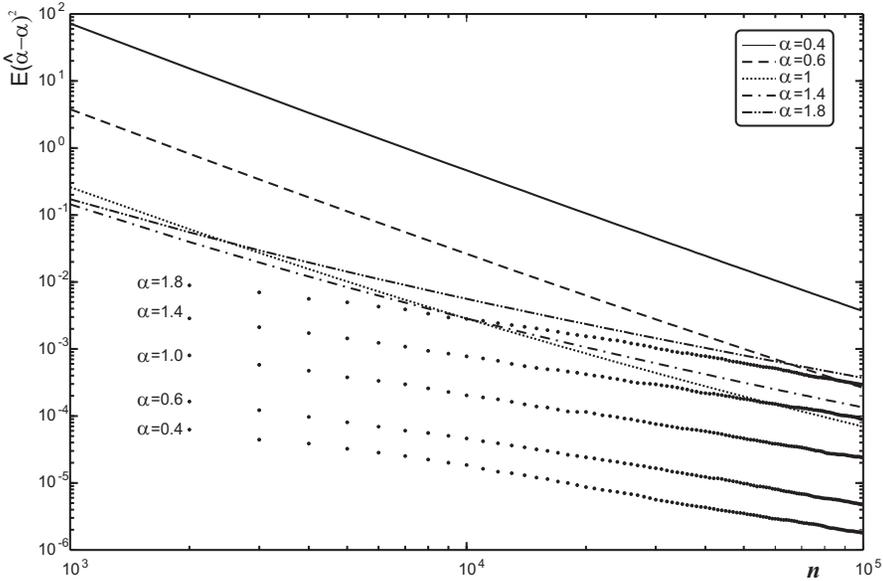}\\
\caption{Evaluation of the variance of the parameter $\hat\alpha$ estimate.
Dots denote selected values of the parameter $\alpha$. The curves denote the results of calculation by formula
(\ref{eq:alpha_vat_main}).\label{fig:alpha_var}}
\end{figure}

By substituting this estimate into (\ref{eq:alpha_expansion1}), we derive the
following inequality
$$
|\tilde\alpha-\alpha|\leqslant\frac{|\tilde\nu-\nu|}{2\nu^{3/2}}+12(\tilde\nu-\nu)^2.
$$
Next, by squaring both sides of the inequality, we
get
\begin{equation}\label{eq:alpha_var_est}
\mean(\tilde\alpha-\alpha)^2\leqslant\frac{\mean(\tilde\nu-\nu)^2}{4\nu^{3}}+
\frac{12}{\nu^{3/2}}\mean|\tilde\nu-\nu|^3+144\mean\left(\tilde\nu-\nu\right)^4.
\end{equation}
By substituting (\ref{eq:nu_2moment}), (\ref{eq:nu_3moment}) and $(\ref{eq:nu_4moment})$
into (\ref{eq:alpha_var_est}) we get the theorem.

The results of the theorem were used for calculation of sample variance of the parameter $\alpha$.
Algorithm of \cite{Kanter:AnnProbab:1975} was used for the modelling of the stable random variables. Using this
algorithm, we simulated the sample $X_1,X_2,\dots,X_n$ of the independent stable distributed random
variables with predefined values of the parameters $\alpha,\beta,\theta,\gamma$. We constructed the
estimate $\tilde\alpha$ and calculated the empirical variance of this estimate. This results of calculation
are presented in Fig.~\ref{fig:alpha_var}. Here dots denote the empirical variance of the estimate $\tilde\alpha$
for the values of this parameter shown in the figure. Curved liens denote the results
given by Theorem~\ref{theo:alpha_var}. This figure illustrates that the increasing of the
parameter $\alpha$ leads to improved variance estimate, obtained from Theorem~\ref{theo:alpha_var}.

\section{Approximation of the local fluctuation fluxes by the stable laws}

In the investigation of local fluctuation fluxes in the edge region of plasma, investigators face
the fact that the distribution densities of the amplitude of measurable fluctuation have heavy tails and
acute vertices. This fact is described in many work (see, e.g. \cite{Carreras:PhysRevLett:99,Sattin:JPhysicsConfSeries:05,
Hidalgo:PlasmaPhysControlFusion:02,Skvortsova:PlasmaPhysRep:05,Carreras:PhysPlasmas:96,
Budaev:NuclFusion:06,Zweben:PlasmaPhysControlFision:07}). As a rule, these works state the fact
on the presence of heavy tails of distributions, study the kurtosis and skewness coefficients \cite{Labit:PhysRevLett:07}, and
some of them make some assumptions on self-similarity of distributions obtained by different
facilities \cite{Pedrosa:PhysRevLett:99}. The presence of heavy tails and self-similarity of distributions
leads to the idea of applying stable laws to approximation of their description.

\begin{floatingfigure}{4cm}\centering
\includegraphics[width=0.3\textwidth]{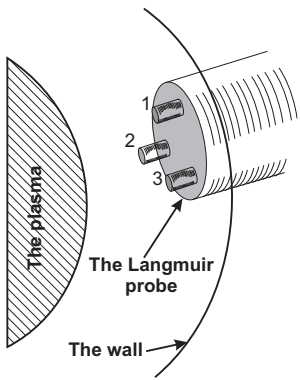}
\caption{\small The experiment scheme in the fluctuation
flux mesuarment.}\label{fig:probe} %
\end{floatingfigure}

A local fluctuation flux in the radial direction is defined as
\begin{equation}\label{eq:flux}
\tilde\Gamma=\delta n_e\delta v_r,
\end{equation}
where $\delta n_e$ are fluctuations of plasma densities, $\delta v_e=c\delta E_\theta/B$ are
fluctuations of radial velocity, $\delta E_\theta=(\delta\varphi_1-\delta\varphi_2)/\Delta\theta r$ are
fluctuations of a poloidal electric field, $\delta\varphi$ are fluctuations  of floating potential,
$\theta$ is a poloidal coordinate, and $r$ is the mean radius of magnetic surface.

Measurements were made on the stellarator L-2M. This is a dual-lead stellarator, the major radius
of the torus $R=100$ cm., and the mean minor radius is $\langle r\rangle=11.5$ cm. The plasma was produced and
heated by  electron-cyclotron heating (ECH) regime at the second harmonic gyromagnetic
frequency of the electron. The magnetic filed in the central region of the plasma was $B=1.3 - 1.4$ T.
The gyrotron radiation power was $P_0=150 - 200$ kW, and the duration of the microwave impulse was
10-12 msec. The measurements were made in the plasma with average densities
$\langle n\rangle=(1.3-1.8)\cdot 10^{13}\ cm^{-3}$, and the central electron temperature was
$T_e(0)=0.6 - 1.0$ keV. The hydrogen was used as a plasma-forming gas. In the edge plasma, with the radius
$r/r_s=0.9$ (where $r_s$ is a separatrix radius), the density was at the level of $n(r)=(1-2)\cdot10^{12}\ cm^{-3}$,
the electron temperature was $T_e(r)=30-40$, the relative density fluctuation range in the edge region of
plasma was $(\delta n/n)_{\scriptsize out}=0.2 - 0.25$, and in the inner plasma region it was at the level
$(\delta n/n)_{\scriptsize in}=0.1$.

\begin{figure}[t]
\includegraphics[width=0.9\textwidth]{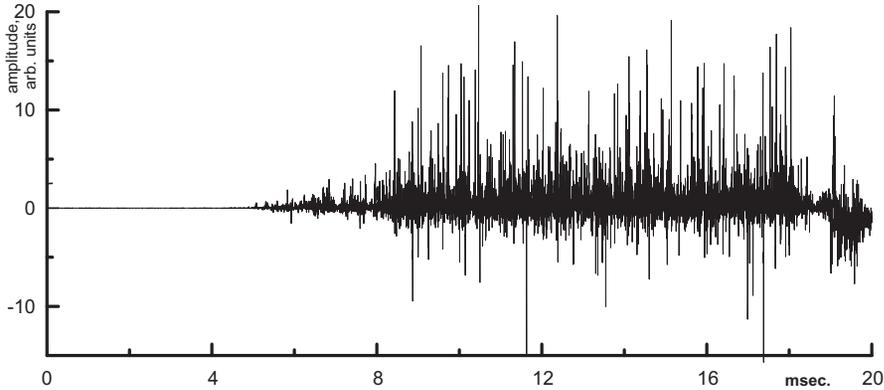}\\
\caption{\small The typical signal of the local fluctuation
flux in the radial direction.\label{fig:fluxsignal}}
\end{figure}

The local fluctuation flux measurements were made by means of a three-pin Langmuir probe. The
experiment scheme is presented in Fig.~\ref{fig:probe}. The probe is located  in a peripheric
region of the plasma. The floating potential fluctuations $\delta\varphi_1$  and $\delta\varphi_2$
at the two neighboring points are measured by probes 1 and 2. The plasma density fluctuation $\delta n_e$
is measured by probe 3. Next, the local fluctuation flux is calculated by formula (\ref{eq:flux}).
A typical  measuring signal is showed in Fig.~\ref{fig:fluxsignal}. The time from the beginning
of the signal is put off in the X-direction, and the time is measured in msec. The plasma appears
in 8 msec. and disappear in 16 msec. We shall investigate  namely this time interval (from 8 msec. to 16 msec.).

\begin{floatingfigure}{5.5cm}
\centering
\includegraphics[width=0.45\textwidth]{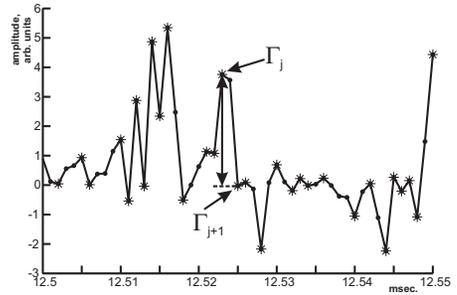}\\
\caption{\small The signal is showed in Fig.~\ref{fig:fluxsignal}
and plotted in the limits 12.5~-~12.55 msec.}\label{fig:fluxundig}
\end{floatingfigure}

Before applying the algorithm in statistical parameter estimation,  we must be sure of independence
of random variables that form a sample. Let us analyse the signal under consideration on an enlarged scale.
The investigated signal is showed in Fig.~\ref{fig:fluxundig} in the time interval from 12.5 msec. to 12.55 msec.
Here dots denote flux values, which were written by the analog-digital convertor. The signal was quantized
at the frequency 1 MHz. The figure illustrates that despite that the signal is very irregular, the
trace of the increasing and decreasing intervals of the signal is considered. It means that the function
depicted in Fig.~\ref{fig:fluxsignal} has a derivative, and, consequently, the successive points
that describe an increase and a decrease in the flux are not independent. The dependence of
investigated values does not allow us to use the algorithm of statistical estimation of parameters,
because the proposition on statistical independence of random variables that form a sample, makes
the basis of this algorithm. In order to use the algorithm, we have to exclude dependent values from the
sample.

We proceed in the following way. Let us have a time series of the flux values $\tilde\Gamma_i,\ i=1,\dots, N$,
where $N$ is the total number of points in the signal. We select from the time series only those values
which correspond to a maximum or a minimum of the flux. As a result, we get another time series which we denote
as $\tilde\Gamma_j,\ j=1,\dots, n$ and $n<N$. In Fig.~\ref{fig:fluxundig}, these points are denoted
by an asterisk. We exclude all the other points from consideration. Next, we calculate increments between
two successive values of the maximum and minimum $\Delta\tilde\Gamma_k=\tilde\Gamma_{j+1}-\tilde\Gamma_j$,
where $k=1,\dots, n-1$.

Thus, we got a sample $\tilde\Gamma_k$. We suppose that the random variables in the sample
$\tilde\Gamma_k$ are independent and have a stable distribution with the characteristic exponent $\alpha$.
Let us apply the algorithm of statistical parameter estimation to this sample, which was described
in section~\ref{sec:est}, and compare the probability density function of this sample
with the probability density function of the stable law with the parameters
$\tilde\alpha, \tilde\beta,\tilde\lambda$. The results of comparison are presented in Fig.~\ref{fig:shotpdf}.
Here, the solid line denotes the density of stable law. As in the estimation we have got the value of the
parameter $\beta=0$, we used the algorithm described in \cite{Uchaikin:JMathSci:02} for
 calculating the density of stable law. The figures illustrates that a slope of the tail of the stable law
as $x\to\infty$ is coincident with the slope of the tail of the experimental density for flux amplitudes.
It means that  the characteristic exponent of the stable law is estimated correctly. Also, we can
 infer from the form of densities that the increments of amplitudes of local fluctuation
 fluxes in the peripheral region of plasma can be described by stable laws.

\begin{figure}
\includegraphics[width=0.45\textwidth]{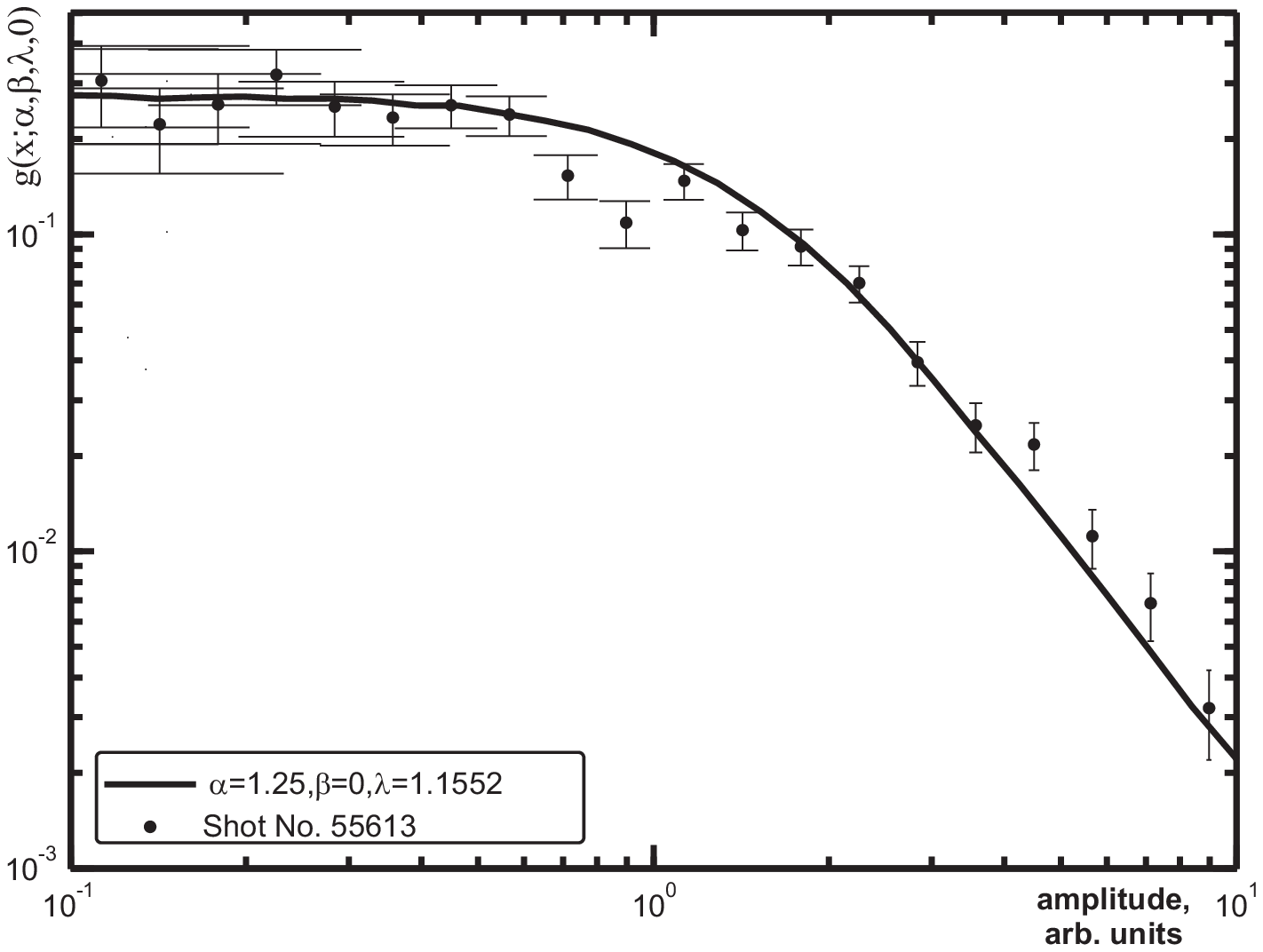}\hfill
\includegraphics[width=0.45\textwidth]{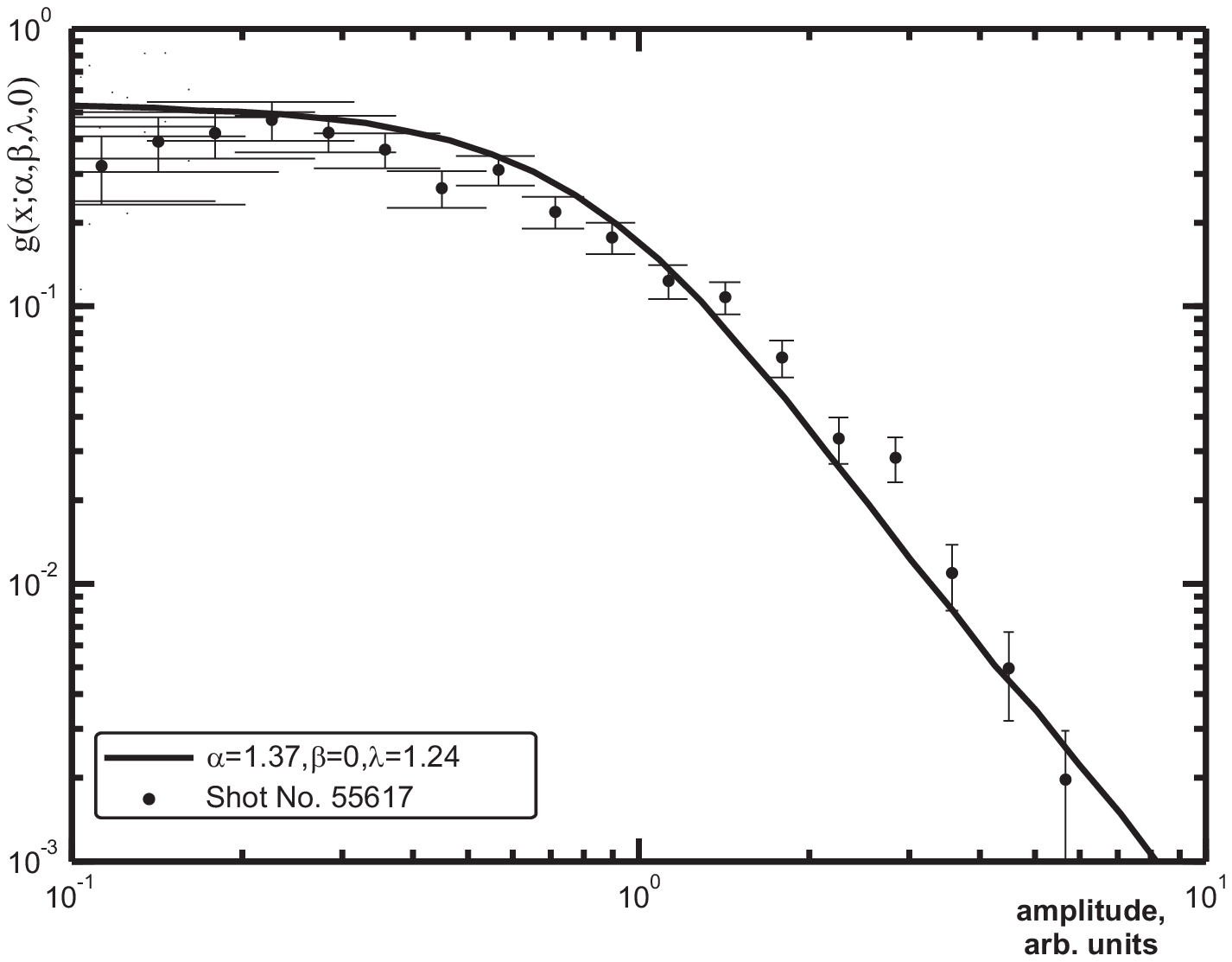}\\
\caption{\small Description of the probability density function
of amplitudes of fluctuation fluxes by stable laws for discharge No. 55613 (on the left)
and No. 55617 (on the right). Here dots denote the empirical density of amplitude increments
of fluctuation fluxes, solid lines denote the theoretical density of the stable law.\label{fig:shotpdf}}
\end{figure}

 \section{Conclusion}

The estimate of a root-mean-square deviation of the characteristic exponent of stable law has been found.
The comparison of the obtained estimation of root-mean-square deviation with the sample variance
of the parameter $\alpha$ has showed that, by increasing the value of
the parameter  $\alpha$ the estimate of root-mean-square deviation is improved.
It is possible to use this estimate for calculating of the variance of the parameter $\alpha$ with the values
$\alpha\geqslant 1.5$.

Also stable laws were used to describe fluctuation fluxes in the peripherical region of plasma.
The reason for their application  is the fact that experimental distribution densities of the flux
amplitude have heavy tails and sharp peaks. Therefore the normal distribution could not
describe these distributions. The calculation has shown that probability
density distributions of amplitudes of fluctuation fluxes are well described by stable laws.
The coincidence  of the tail slope of empirical and theoretical densities may serve as a
corroboration of this result. Application of $\xi^2$ or Kolmogorov-Smirnov goodness-of-fit tests
is rather difficult. The point is that densities of stable laws are calculated by the Monte-Carlo method
and therefore the density values contain a statistical error, which has a negative influence on the
result of the test.

\end{document}